\documentclass[12pt]{article}

\usepackage{mathrsfs}
\usepackage{amsmath}
\usepackage{amsfonts}
\usepackage{amssymb}

\usepackage{amsmath, amsthm, amsfonts, amssymb}
\usepackage{comment}

\newcommand{\CC}{{\mathbb C}}

\newcommand{\HH}{{\mathbb H}}

\newcommand{\PP}{{\mathbb P}}

\def \-{\bar}

\newtheorem{theorem}{Theorem}[section]
\newtheorem{lemma}[theorem]{Lemma}
\newtheorem{corollary}[theorem]{Corollary}

\newtheorem{question}[theorem]{Question}

\newtheorem{example}[theorem]{Example}
\newtheorem{remark}[theorem]{Remark}
\newtheorem{claim}[theorem]{Claim}
\date{}

\begin{document}

\title{\bf Umehara algebra and complex submanifolds of indefinite complex space forms
}

\author{
\ \ Xu Zhang\footnote{Corresponding author.}
, \ \ Donghai Ji
}

\vspace{3cm} \maketitle

\begin{abstract}
The Umehara algebra is studied with motivation on the problem of the non-existence of common complex submanifolds. In this paper, we prove some new results in Umehara algebra and obtain some applications.
In particular, if a complex manifolds admits a holomorphic polynomial isometric immersion to one indefinite complex space form, then it cannot admits a holomorphic isometric immersion to another  indefinite complex space form of different type. Other consequences include the non-existence of the common complex submanifolds for indefinite complex projective space or hyperbolic space and a complex manifold with a distinguished metric, such as homogeneous domains, the Hartogs triangle, the minimal ball, the symmetrized polydisc, etc, equipped with their intrinsic Bergman metrics, which generalizes more or less all existing results.
\end{abstract}

\bigskip
\section{Introduction}

The study of holomorphic isometric embedding between complex manifolds is a classical problem in complex differential geometry.
Starting with Bochner's paper \cite{[B]}, such questions have been studied extensively by
many authors (e.g. \cite{[C], [CU], [DL2], [HY1], [LZ], [Mo2], [Mo3], [U1], [YZ]}.) In his PhD. Thesis \cite{[C]}, E. Calabi obtained
the existence, uniqueness and global extension of a local holomorphic
isometry from a complex K\"ahler manifold into a complex space form.
In particular, Calabi proved that any complex space form
cannot be locally isometrically embedded into another complex space form
with a different curvature sign with respect to their canonical
K\"ahler metrics.
The key idea in Calabi's work is his diastasis function, which is a Hermitian symmetric, potential function of the K\"ahler metric that does not have the pluriharmonic terms. Using the diastasis function, Calabi was able to reduce the metric tensor equation to
functional identities involving diastasis functions. \\

Using the diastasis function, Umehara studied an interesting question
whether two complex space forms can share a common complex submanifold with induced metrics and he proved that
two complex space forms with different curvature signs cannot share
a common K\"ahler submanifold  \cite{[U1]}. 
Umehara later  defined the so called Umehara algebra and generalized  Calabi's existence and uniqueness results for holomorphic isometric embeddings from a complex manifold with an indefinite K\"ahler metric into an indefinite complex space form \cite{[U2]}. \\

Two K\"ahler manifolds $M_1, M_2$ share a
common complex  submanifold if a complex submanifold of $M_1$ endowed with the induced metric is biholomorphically isometric to a complex submanifold of $M_2$ endowed with the induced metric as well.
Di Scala and Loi showed in \cite{[DL2]} that Hermitian symmetric spaces of compact type and of non-compact type do not share common complex submanifolds.
In addition, the fact that Euclidean spaces and Hermitian symmetric spaces of compact types do not share common complex submanifolds follows from Umehara's result \cite{[U1]} and the classical
Nakagawa-Takagi embedding of Hermitian symmetric spaces of compact type into complex projective spaces.
Finally, it was shown by Huang and Yuan in \cite{[HY2]} that Euclidean spaces and Hermitian symmetric spaces
of non-compact types do not share common complex submanifolds.
In \cite{[CDY]}, Cheng, Di Scala and Yuan generalized the problem to indefinite complex space forms. In particular, they proved that the indefinite Euclidean space cannot share a common complex submanifold with an indefinite complex projective space or an indefinite complex hyperbolic space. However, whether or not an indefinite complex projective space and an indefinite complex hyperbolic space can share a common complex submanifold is an interesting open problem \cite{[CDY], [Y1]}. Such problem of existence/non-existence of common complex submanifold for various different complex manifolds remains an active research problem lately and is studied extensively by different authors (cf. \cite{[CH],[CN],[LM1],[Mos], [STT],[ZJ]}).
\\

The powerful method to attack such problem is to use Umehara algebra. As developed in \cite{[U2]}, if one can prove that certain function does not belong to the Umehara algebra or its quotient field, the negative answer to the problem thus follows. Umehara algebras are intensively studied in \cite{[CDY],[LM1]} and interesting common complex submanifolds problem are discussed there as their consequences.
\\

In this paper, we further consider Umehara algebra and one key feature is that it may involve non-Nash algebraic functions.
The main theorem in section 2 describes that certain real analytic functions are not contained in the Umehara algebra or its field of fractions (cf. Theorem \ref{UA}).
 On one hand, such Umehara algebra leads to a partial answer to the problem of the common complex submanifolds for indefinite complex projective and hyperbolic space. More precisely, we prove that, if a K\"ahler manifold admits a holomorphic polynomial isometric immersion into one indefinite complex space form, then it cannot admits a holomorphic isometric immersion into another indefinite complex space form of different type (cf. Corollary \ref{UA2}).
On the other hand, it has deep consequences in the problem of the common complex submanifolds for indefinite complex space forms and a complex manifold with a distinguished metric. In particular, we are able to provide the sufficient condition for a K\"ahler manifold that does not share common  complex submanifolds with an indefinite complex space form (cf. Corollary \ref{suf}).
The examples include bounded homogeneous domains, the Hartogs triangle, the minimal ball, the symmetrized polydisc, etc, equipped with their intrinsic Bergman metrics, which generalizes more or less all existing results.
The method in this paper is developed from the ideas of Huang and Yuan in \cite{[HY1],[HY2]} for Nash algebraic functions and we now further generalize it to certain non-Nash algebraic functions.
\\

The paper is organized as follows: in section 2, we state the main theorem in the context of the Umehara algebra and give the proof; in section 3, the applications of the holomorphic isometric embeddings are provided. In particular, we show that if a K\"ahler manifold admits a holomorphic polynomial isometric immersion into one indefinite complex space form then it cannot admits a holomorphic isometric immersion into another indefinite complex space form of different type. Moreover, we give the sufficient condition for a K\"ahler manifold that does not share common  complex submanifolds with an indefinite complex space form.

\bigskip

{\bf Acknowledgement:} We thank the referees for very helpful suggestions and comments.

\section{The Umehara algebra}
\subsection{The statement of the main theorem}

Umehara introduces in \cite{[U2]}  an associate algebra on a complex manifold $M$ and uses the algebra to study the holomorphic isometric embedding between complex manifolds. Since the interest here is the local existence of a complex submanifold at some point $p$ in $M$, we modify Umehara's definition as follows. Use $\mathcal{O}_p$ to denote the local ring of germs of holomorphic functions at $p$ and let $\tilde{\mathcal{O}}_p = \{\chi \in \mathcal{O}_p | \chi(p)=0\}$.

Define 
$$ \Lambda_p := \left\{ f | f = \sum_{j=1}^{\kappa} a_j |\chi_j|^2 , \chi_j \in \mathcal{O}_p , a_j \in \mathbb{R} \right\} , $$
and
let $K_p$ be the field of fractions of $\Lambda_p$. It is shown in \cite{[U2]} (cf. Theorem 3.2 in \cite{[U2]}) that every $f \in \Lambda_p$ can be written as $f= h + \overline{h} + \sum_{j=1}^{\kappa} a_j |\chi_j|^2 $ for $h \in\mathcal{O}_p,   a_j \in \mathbb{R}$ and linearly independent $\chi_1, \cdots, \chi_\kappa \in \tilde{\mathcal{O}}_p $.
Moreover, define
$$ \tilde\Lambda_p := \left\{ f = a_0 + \sum_{j=1}^{\kappa} a_j |\chi_j|^2 \in \Lambda_p |  \chi_j \in\tilde{\mathcal{O}}_p,   a_j \in \mathbb{R} \right\}.$$
Let $\prod_{j=1}^t \Lambda_p^{\mu_j} = \left\{ \prod_{j=1}^t h_j^{\mu_ j} | h_1, \cdots, h_t \in \Lambda_p \right\}$ for $\mu_1, \cdots, \mu_t \in \mathbb{R}$.
In particular, $\Lambda_p^{1} \cdot \Lambda_p^{-1} = K_p$ is  the field of fractions of $\Lambda_p$.
Note that the germs of real numbers, denoted by $\mathbb{R}_p$, belong to $\tilde\Lambda_p$.

\medskip

The main theorem of the paper is as follows.

\begin{theorem}\label{UA}
Let $p$ be a fixed point on a complex manifold $M$. Let $\chi_1, \cdots, \chi_m \in \mathcal{O}_p$ be non-constant germs of holomorphic functions at $p$ and denote $\chi=(\chi_1, \cdots, \chi_m)$.  
  The following statements hold:

\begin{itemize}

\item[(i)] If   $\chi_1, \cdots, \chi_m \in \tilde{\mathcal{O}}_p$ are germs of holomorphic polynomials,    $a_1, \cdots, a_m$ are non-zero real numbers and $\alpha<0$, then 
$$\left(1+ \sum_{i=1}^m a_i |\chi_i|^2\right)^{\alpha} \not\in \Lambda_p \setminus \mathbb{R}_p.$$

\item[(ii)] Let  $\chi_1, \cdots, \chi_m \in \tilde{\mathcal{O}}_p$ and $H(z, \overline w) = R\left( \prod_{j=1}^{\kappa_1} H_j^{\mu_j}(z, \overline{w})  \right) \prod_{j=\kappa_1+1}^{\kappa_1+\kappa_2}H_j^{\mu_j}(z, \overline w) $, where \\
$H_1(z, \overline{w}), \cdots, H_{\kappa_1+\kappa_2}(z, \overline{w}) $ are Hermitian symmetric, Nash algebraic functions and do not have non-constant pluriharmonic terms, $R(\cdot)$ is a rational function with real coefficients and $\mu_1, \dots, \mu_{\kappa_1+\kappa_2 } \in \mathbb{R}$. Then

\begin{equation}\label{I}
 \exp \left( H(\chi, \overline{\chi}) \right) \not\in \Lambda_p^{\mu} \setminus \mathbb{R}_p  , 
 \end{equation} and
 \begin{equation}\label{II}
 \log \left( H(\chi, \overline{\chi}) \right) \not\in \tilde\Lambda_p^{\mu} \setminus \mathbb{R}_p .
 \end{equation}

\item[(iii)] Let   $H(z, \overline{w}) $ be a Hermitian symmetric, Nash algebraic function. Then 
\begin{equation}\label{III}
 \exp \left( H(\chi, \overline{\chi}) \right) \not\in  \tilde\Lambda_p^{\mu} \setminus \mathbb{R}_p , 
 \end{equation} and
\begin{equation}\label{IV}
\log H(\chi, \overline{\chi})  \not\in \tilde\Lambda^\mu_p \setminus \mathbb{R}_p.
 \end{equation}
If, in addition,  $H(z, \overline{w}) $ does not have non-constant pluriharmonic terms and $\chi_1, \cdots, \chi_m \in \tilde{\mathcal{O}}_p$, then 
\begin{equation}\label{V}
 \exp \left( H(\chi, \overline{\chi}) \right) \not\in  \prod_{j=1}^\kappa  \Lambda_p^{\nu_j} \setminus \mathbb{R}_p , 
 \end{equation} and
\begin{equation}\label{VI}
\log H(\chi, \overline{\chi})  \not\in \prod_{j=1}^\kappa \tilde{\Lambda}_p^{\nu_j} \setminus \mathbb{R}_p.
 \end{equation}

\end{itemize}
\end{theorem}

\subsection{Proof of the main theorem}

We first prove Part (i) and Part (ii), (iii) will be proved via a different method.

For Part (i), we argue by contradiction and suppose $\left(1+ \sum_{i=1}^m a_i |\chi_i|^2\right)^{\alpha} \in \Lambda_p \setminus \mathbb{R}_p$. Namely, there exists an open neighborhood $U \subset M$ of $p$,  $h \in \mathcal{O}_p$ and  linearly independent $g_1, \cdots, g_n \in \tilde{\mathcal{O}}_p$ such that $$\left(1+ \sum_{i=1}^m a_i |\chi_i|^2\right)^{\alpha} = h + \overline{h} + \sum_{j=1}^r |g_j|^2 - \sum_{j=r+1}^n |g_j|^2$$ on $U$. We choose a holomorphic coordinate $\{z\}$ on $U$ with $z(p)=0$.
By the standard argument to get rid of $\partial\overline\partial$ (cf. \cite{[CU],[Mo2]}), since there is no non-constant pluriharmonic term on the left hand side, it follows that $h + \overline{h}=1$.

First of all, we will show that each $g_j$ can be written as a polynomial of $\chi_1, \cdots, \chi_m$.
By polarization, the identity above is equivalent to

\begin{equation}\label{polar1}
\left( 1+\sum_{i=1}^m a_i \chi_i(z) \overline \chi_i(w) \right)^\alpha = 1+ \sum_{i=1}^r  g_i(z) \overline g_i(w)-\sum_{j=r+1}^n g_j(z) \overline g_j(w),
\end{equation}
where $(z, w) \in U\times\hbox{conj}({U})$, $\hbox{conj}({U})=\{z \in \CC | \overline z \in U\}$, and $\overline{\chi}_i(w) = \overline{\chi_i(\overline{w})}$.
\\  \
Taking $l$-th derivative of  the equation (\ref{polar1}) in $w$ for $l=1,2,\cdots$, and then evaluating at $w=0$, we have the following matrix equation:
\begin{equation}\notag
P = A \cdot Q,
\end{equation}
where
\begin{equation}\notag
A=\begin{bmatrix}
 \cdots \frac{\partial \overline {g}_{i}}{\partial w}(0) \cdots -\frac{\partial \overline {g}_{j}}{\partial w}(0) \cdots \\
\vdots\\
 \cdots \frac{\partial^l \overline {g}_{i}}{\partial w^l}(0) \cdots -\frac{\partial^l \overline {g}_{j}}{\partial w^l}(0) \cdots \\
\vdots\\
\end{bmatrix}_{\infty\times n,} ~
Q=\begin{bmatrix}
\vdots\\ g_i(z) \\ \vdots \\g_j(z) \\ \vdots
\end{bmatrix}_{n \times 1,}
\text{and}~
P=\begin{bmatrix}
 \vdots\\ p_j \\  \vdots
\end{bmatrix}_{\infty\times1,}\end{equation} with each $p_k$ being polynominal function in $\chi_1(z),\cdots,\chi_m(z)$. 

Now we show rank$(A)=n$.
Suppose rank$(A)=d< n$. Without loss of generality, we assume that the first $d $ columns are linearly independent in the
 coefficient matrix $A$, denoted by $L_1,L_2,\cdots,L_d$. Then, for any $d'$ with $d< d' $, the $d'$-th columns is linear combination of $L_1,L_2,\cdots,L_d$, i.e.
\begin{equation}\notag\label{function1}
 L_{d'}=\sum_{i=1}^{d}C_i L_i.
\end{equation}
In other words, $\overline g_{d'}$ can be written as linear combination of $\{ \overline g_1, \cdots, \overline g_d\}$ by the Taylor expansion, meaning that $\{g_1, \cdots, g_n\}$ is not linear independent. This is a contradiction.

Since rank$(A)=n$, there exist $k=k_1, \cdots, k_n$ such that $k_1$-row to $k_n$-row in matrix $A$ are linearly independent and all other rows can be written as linear combinations of $k_1$-row up to $k_n$-row. Re-organize the matrices $A$ and $P$ by deleting rows other than $k_1$-row to $k_n$-row, and denote the corresponding matrices by $A_n, P_n$ respectively. We obtain the matrix equation $$P_n = A_n \cdot Q.$$ Since $A_n$ is nondegenerate, $Q = A_n^{-1} P_n$. In other words, each $g_i, g_j$ can be written as the linear combination of these $p_j$. As a consequence, they are still polynomials functions in $\chi_1(z),\cdots,\chi_m(z)$ and thus are  polynomials as well since $\chi_1, \cdots, \chi_m$ are polynomials.

Therefore, we reach a contradiction to (\ref{polar1}) as the right hand side is a polynomial while the left hand side is not. This completes the proof of Part (i).

\bigskip

If $H(\chi, \overline\chi)$ is constant, then
Part (ii), (iii) are trivial. Otherwise, we also argue by contradiction.  Suppose the conclusion of Part (ii) and (iii) is false. Namely, there exist an open neighborhood $U$ of $p$, $a_j \in \mathbb{R}$ and $g_1, \cdots, g_k, h_1, \cdots, h_l \in \mathcal{O}_p$ or $\tilde{\mathcal{O}}_p$ accordingly, such that 

\begin{equation}\label{eq2}
\exp \left( H(\chi, \overline\chi) \right) = \prod_{j=1}^\kappa \left(a_j+ \sum_{\alpha=1}^{k_j} |g_{\alpha}|^2 - \sum_{\beta=1}^{l_j} |h_{\beta}|^2 \right)^{\nu_j}
\end{equation}
and

\begin{equation}\label{eq3}
\log \left( H(\chi, \overline\chi) \right) = \prod_{j=1}^\kappa \left( a_j + \sum_{\alpha=1}^{k_j}  |g_{\alpha}|^2 - \sum_{\beta=1}^{l_j} |h_{\beta}|^2 \right)^{\nu_j} 
\end{equation}
on $U$. We again choose a holomorphic coordinate $\{z\}$ on $U$ with $z(p)=0$.
We will need the following algebraicity result as in \cite{[LM1]} (cf. Lemma 2.3 in \cite{[ZJ]}). 

\begin{lemma}\label{algebra}
Writing $S= \{\phi_1, \cdots, \phi_{k+l+m} \}= \{\chi_1, \cdots, \chi_m, g_1, \cdots, g_{k}, h_1, \cdots, h_l \}$, then there exists a maximal algebraic independent subset $\{\phi_1, \cdots, \phi_d\} \subset S$  over the field $\cal{R}$ of rational functions on $U$, and Nash algebraic functions $\hat\phi_j(t, X_1, \cdots, X_d)$ defined in a neighborhood $\hat U$
of $\{(s, \phi_1(s),\cdots , \phi_d(s))| s \in U\}$,  such that $\phi_j(t) = \hat\phi_j(t, \phi_1(t), \cdots, \phi_d(t))$ for all $1 \leq j \leq k+l+m$ after
shrinking $U$ toward the origin if needed.
\end{lemma}

By polarization as in (\ref{polar1}), equations (\ref{eq2}), (\ref{eq3}) are equivalent to
\begin{equation}\label{polar2}
\begin{split}
\exp \left( H(\chi(z), \overline\chi(w)) \right) &= \prod_{j=1}^\kappa \left(a_j+ \sum_{\alpha=1}^{k_j} g_{\alpha}(z) \overline{g}_{\alpha}(w) - \sum_{\beta=1}^{l_j} h_{\beta}(z) \overline{h}_{\beta}(w) \right)^{\nu_j}\\
\end{split}
\end{equation}
and
\begin{equation}\label{polar3}
\begin{split}
\log \left( H(\chi(z), \overline\chi(w)) \right) &= \prod_{j=1}^\kappa \left(a_j + \sum_{\alpha=1}^{k_j} g_{\alpha}(z) \overline g_{\alpha}(w) - \sum_{\beta=1}^{l_j} h_{\beta}(z) \overline h_{\beta}(w) \right)^{\nu_j}\\
\end{split}
\end{equation}
for $(z, w) \in U\times\hbox{conj}({U})$. 

Denote $X=(X_1, \cdots, X_d)$. For $(s, X, t) \in \hat U \times\hbox{conj}({U}) $, now define
\begin{equation}\label{A}
\begin{split}
\Psi_1(s, X, t)=H(\hat\chi(s, X), \overline\chi(t)) -  \sum_{j=1}^\kappa \nu_j \log \left( a_j+ \sum_{\alpha=1}^{k_j} \hat g_{\alpha}(s, X) \overline g_{\alpha}(t) - \sum_{\beta=1}^{l_j} \hat h_{\beta}(s, X) \overline h_{\beta}(t)  \right),
\end{split}
\end{equation}

\begin{equation}\label{B}
\Psi_2(s, X, t) = \log H(\hat\chi(s, X), \overline\chi(t)) - \prod_{j=1}^\kappa \left(a_j+\sum_{\alpha=1}^{k_j} \hat g_{\alpha}(s, X) \overline g_{\alpha}(t) - \sum_{\beta=1}^{l_j} \hat h_{\beta}(s, X) \overline h_{\beta}(t) \right)^{\nu_j}.
\end{equation}



\begin{claim}
Under the corresponding assumptions, $\Psi_j^{(l)}(s, X):=\frac{\partial ^l}{\partial t^l }\Psi_j (s, X, 0)$ is Nash algebraic for any $j=1, 2$ and $l \in \mathbb{N}$.
\end{claim}
For $j=1$, since, under the assumption of Theorem \ref{UA}.(ii), $H_1(z, 0), \cdots, H_{\kappa_1+\kappa_2}(z, 0)$ are all constants,  $\frac{\partial ^l}{\partial t^l } H(\hat\chi(s, X), \overline\chi(0))$ is Nash algebraic. 
Under the assumption of Theorem \ref{UA}.(iii), $\frac{\partial ^l}{\partial t^l } H(\hat\chi(s, X), \overline\chi(0))$ is obviously Nash algebraic. Thus $\Psi_1^{(l)}(s, X)$ is Nash algebraic. For $j=2$, under the assumption of Theorem \ref{UA}.(ii), $\frac{\partial ^l}{\partial t^l } \log H(\hat\chi(s, X), \overline\chi(0))$ is Nash algebraic by the same argument. Under the assumption of Theorem \ref{UA}.(iii), $\frac{\partial ^l}{\partial t^l } \log H(\hat\chi(s, X), \overline\chi(0))$ is also obviously Nash algebraic. The other term is also Nash algebraic by noting $a_j+\sum_{\alpha=1}^{k_j} \hat g_{\alpha}(s, X) \overline g_{\alpha}(0) - \sum_{\beta=1}^{l_j} \hat h_{\beta}(s, X) \overline h_{\beta}(0)=a_j$. Thus $\Psi_2^{(l)}(s, X)$ is Nash algebraic.

\begin{lemma}
For any $j=1, 2$ and $l >0$,
$$\Psi_j^{(l)}(s, X) \equiv 0.$$
\end{lemma}

\begin{proof}
Suppose that $\Psi_j^{(l)} (s, X, 0)$ is  not constant.
There exists a holomorphic polynomial $P(s, X, y)=A_{\hat d}(s, X)y^{\hat d} + \cdots + A_0(s, X)$ of degree $\hat d$ in $y$, with $A_0(s, X) \not\equiv0$ such that $P(s, X, \Psi_j^{(l)}(s, X, 0)) \equiv 0$.
As $\Psi_j(s, \phi_1(s), \cdots, \phi_d(s), t) \equiv 0$ for any $(s, t) \in {U}\times\hbox{conj}({U})$, it
follows that $\Psi_j^{(l)}(s, \phi_1(s), \cdots, \phi_d(s)) \equiv 0$ and
therefore $A_0(s, \phi_1(s), \cdots, \phi_d(s)) \equiv 0$. This means that
$\{\phi_1(s), \cdots, \phi_d(s)$\} are algebraic dependent over
$\cal{R}$. This is a contradiction and it follows that $\Psi_j^{(l)} (s, X)$ is a constant. 
Therefore, $\Psi_j^{(l)} (s, X) = \Psi_j^{(l)}(s, \phi_1(s), \cdots, \phi_d(s)) \equiv 0.$
\end{proof}

By the Taylor expansion, it follows that $\Psi_j (s, X, t) \equiv \Psi_j (s, X, 0)$ for any $t$ near 0. We now obtain:
$$\Psi_1(s, X, t)=  H(\hat\chi(s, X), \overline\chi(0))-  \sum_{j=1}^\kappa \nu_j \log \left(\sum_{\alpha=1}^{k_j} \hat g_{\alpha}(s, X) \overline g_{\alpha}(0) - \sum_{\beta=1}^{l_j} \hat h_{\beta}(s, X) \overline h_{\beta}(0)  \right),$$
equivalent to 
\begin{equation}\label{AA}
\exp\left\{H(\hat\chi(s, X), \overline\chi(t)) - H(\hat\chi(s, X),\overline\chi(0)) \right\} = \prod_{j=1}^t \left(\frac{a_j+\sum_{\alpha=1}^{k_j} \hat g_{\alpha}(s, X) \overline{g}_{\alpha}(t) - \sum_{\beta=1}^{l_j} \hat h_{\beta}(s, X) \overline{h}_{\beta}(t)}{ a_j+ \sum_{\alpha=1}^{k_j} \hat g_{\alpha}(s, X) \overline{g}_{\alpha}(0) - \sum_{\beta=1}^{l_j} \hat h_{\beta}(s, X) \overline{h}_{\beta}(0)}\right)^{\nu_j};
\end{equation}
$$\Psi_2(s, X, t) =\log H(\hat\chi(s, X), \overline\chi(0)) -  \prod_{j=1}^\kappa a_j^{\nu_j}  ,$$ equivalent to 
\begin{equation}\label{BB}
\exp\left\{   \prod_{j=1}^\kappa  a_j^{\nu_j}   \right\} \frac{H(\hat\chi(s, X), \overline\chi(t))}{H(\hat\chi(s, X), \overline\chi(0))} 
= \exp \left( \prod_{j=1}^\kappa \left(a_j+ \sum_{\alpha=1}^{k_j} \hat g_{\alpha}(s, X) \overline g_{\alpha}(t) - \sum_{\beta=1}^{l_j} \hat h_{\beta}(s, X) \overline h_{\beta}(t) \right)^{\nu_j}\right).
\end{equation}



In order to achieve a contradiction, we need the following lemmas.

\begin{lemma}\label{noncons}
\begin{itemize}
\item 
There exists some $t_1 \in \hbox{conj}({U})$ such that $H(\hat\chi(s, X), \overline\chi(t_1)) $ is not constant in $(s, X) \in \hat U$.

\item 
There exists some $ t_2 \in \hbox{conj}({U})$ such that $\sum_{\alpha} \hat g_{\alpha}(s, X) \overline g_{\alpha}(t_2) - \sum_{\beta} \hat h_{\beta}(s, X) \overline h_{\beta}(t_2)  $ is not constant in $(s, X) \in \hat U$.
\end{itemize}
\end{lemma}

\begin{proof}
Suppose not. We may choose $t_1=t_2=\overline s, X=(\phi_1(s), \cdots, \phi_d(s))$ such that $$H(\hat\chi(s, \phi_1(s), \cdots, \phi_d(s)), \overline\chi(\overline s)) = H(\chi(s), \overline{\chi(s)}) = \text{constant},$$
and 

\begin{equation}\notag
\begin{split}
~& \sum_{\alpha} \hat g_{\alpha}(s, \phi_1(s), \cdots, \phi_d(s)) \overline g_{\alpha}(\overline s) - \sum_{\beta} \hat h_{\beta}(s, \phi_1(s), \cdots, \phi_d(s)) \overline h_{\beta}(\overline s) \\
&= \sum_{\alpha} |g_{\alpha}(s)|^2 - \sum_{\beta} |h_{\beta}(s)|^2 = \text{constant}.
\end{split}
\end{equation}
Both identities contradict to our initial assumptions.
\end{proof}

  \begin{lemma}\label{cont}
  Let $V \subset \CC^k$ be a connected open set. Let $H(\xi), H_1(\xi), \cdots,  H_{\kappa_1+\kappa_2}(\xi)$
     be  holomorphic Nash algebraic functions on $V$, $\mu, \mu_1, \cdots, \mu_{\kappa_1+\kappa_2} \in \mathbb{R}\setminus \{0\}$ and $R(\cdot)$ be a holomorphic rational function.
    Assume that   
  \begin{equation}\label{albert1}
    \exp \left\{ R\left( \prod_{j=1}^{\kappa_1} H_j^{\mu_j}(\xi)  \right) \prod_{j=\kappa_1+1}^{\kappa_1+\kappa_2}H_j^{\mu_j}(\xi)  \right\} =   H^{\mu}(\xi)   ,
   \end{equation}    or
    \begin{equation}\label{albert2}
  R\left( \prod_{j=1}^{\kappa_1} H_j^{\mu_j}(\xi)  \right) \prod_{j=\kappa_1+1}^{\kappa_1+\kappa_2}H_j^{\mu_j}(\xi)  = \exp \left\{ H^\mu(\xi) \right\} , 
   \end{equation}   
   for $\xi\in V$. Then
    $H(\xi)$ is constant.
   \end{lemma}

Lemma \ref{cont} is a further generalization of Lemma 2.1 in \cite{[ZJ]} and can be proved by the similar argument.  The idea is very simple: when the right hand sides approaches the infinity, the  left hand side grow exponentially fast while the right hand side at most polynomially in (\ref{albert1}) and vice versa in (\ref{albert2}). We will not give the detail here.

\medskip

Now let us see  how to reach the contradiction. We first consider cases (\ref{I}),(\ref{III}),(\ref{V}). Under the assumption of Theorem \ref{UA}.(i) or the second case of (ii),
since $H(z, \overline w)$ does not contain the non-constant pluriharmonic terms, $H(\hat\chi(s, X), \overline{\chi}(0))$ is constant. Since the left hand of  (\ref{AA}) is non-constant, we reach the contradiction by (\ref{albert1}),(\ref{albert2}) in Lemma \ref{cont}, respectively. Under the assumption of the first case of (ii) in Theorem \ref{UA}, since $\kappa=1$ and the denominator on the right hand side is constant, the right hand side is non-constant in (\ref{AA}). This is again a contradiction to (\ref{albert1}). Now we consider  cases (\ref{II}),(\ref{IV}),(\ref{VI}). Under the assumption of Theorem \ref{UA}.(i) or the second case of (ii),
since $H(z, \overline w)$ does not contain the non-constant pluriharmonic terms, $H(\hat\chi(s, X), \overline{\chi}(0))$ is constant. Since the left hand of  (\ref{BB}) is non-constant, we reach the contradiction by (\ref{albert1}),(\ref{albert2}) in Lemma \ref{cont}, respectively. Under the assumption of the first case of (ii) in Theorem \ref{UA}, since $\kappa=1$,  the right hand side is non-constant in (\ref{BB}) and we reach a contradiction to (\ref{albert2}).
 This finishes the proof of the theorem.
 \qed

\section{Applications to the holomorphic isometric embedding}

Theorem \ref{UA} has wide applications in the study of non-existence of the common complex submanifolds, that may generalize various results obtained in \cite{[HY2]}\cite{[CDY]}\cite{[CH]}\cite{[ZJ]}. 

\subsection{Indefinite complex space forms}
The indefinite K\"ahler metric $\omega_{\mathbb{C}^{N}_\kappa}$ on the indefinite complex Euclidean space $\mathbb{C}^{N}_\kappa = \mathbb{C}^{N} =\{ (z_1, \cdots, z_N) | z_j \in \mathbb{C}, j=1, \cdots, N\} (0\leq \kappa \leq N)$ is given by 
$$\omega_{\CC_\kappa^N} = 4\sqrt{-1} \partial \overline\partial \left( \sum_{i=\kappa+1}^{N} |z_i|^2 - \sum_{j=1}^{\kappa} |z_{j}|^2\right) .$$ 
The indefinite complex projective space $\CC\PP_\kappa^N(b)(0\leq \kappa \leq N)$ of positive constant holomorphic sectional curvature $b>0$ is the open complex submanifold $\{(\xi_0, \cdots, \xi_N) \in \mathbb{C}^{N+1}| \sum_{i=0}^{N-\kappa} |\xi_i|^2 - \sum_{j=0}^{\kappa-1} |\xi_{N-j}|^2>0 \} / \CC^*$ of $\CC\PP^N$. The indefinite K\"ahler metric of $\CC\PP_\kappa^N(b)$ is given by 
$$\omega_{\CC\PP_\kappa^N(b)} = \frac{4\sqrt{-1}}{b} \partial \overline\partial \log \left( \sum_{i=0}^{N-\kappa} |\xi_i|^2 - \sum_{j=0}^{\kappa-1} |\xi_{N-j}|^2\right) .$$ 
The indefinite complex hyperbolic space
 $ \CC\HH_\kappa^N(b) (0 \leq \kappa \leq N)$ of negative constant holomorphic sectional curvature $b<0$ is obtained from $\CC\PP^N_{N-\kappa}(-b)$ with indefinite K\"ahler metric
 $$ \omega_{\CC\HH_\kappa^N(b)} = - \frac{4\sqrt{-1}}{(-b)} \partial \overline\partial \log \left( \sum_{i=0}^{\kappa} |\xi_i|^2 - \sum_{j=0}^{N- \kappa-1} |\xi_{N-j}|^2\right).$$ Without loss of generality, assuming  $\xi_0 \not=0 $, then under inhomogeneous coordinates $(z_1, \cdots, z_N) = (\xi_1 / \xi_0, \cdots, \xi_N / \xi_0)$, the metrics are given by
$$\omega_{\CC\PP_\kappa^N(b)} = \frac{4\sqrt{-1}}{b} \partial \overline\partial \log \left( 1+ \sum_{i=1}^{N-\kappa} |z_i|^2 - \sum_{j=0}^{\kappa-1} |z_{N-j}|^2\right) $$ and
$$\omega_{\CC\HH_\kappa^N(b)} = - \frac{4\sqrt{-1}}{(-b)} \partial \overline\partial \log \left(1+ \sum_{i=1}^{\kappa} |z_i|^2 - \sum_{j=0}^{N- \kappa-1} |z_{N-j}|^2\right).$$
In particular, when $\kappa=0$, $\CC^{N,0}, \CC\PP_\kappa^N(b), \CC\HH_\kappa^N(b)$ are just the standard complex Euclidean, projective, hyperbolic space, respectively.

\medskip

Furthermore, suppose that $D$ is a complex manifold such that  there exist  holomorphic maps $F=(F_1, \cdots, F_N): D \rightarrow \CC\PP^{N}_\kappa$ and $L=(L_1, \cdots, L_{N'}): D \rightarrow \CC\HH_{\kappa'}^{N'}(b)$) with $F^*\omega_{\CC\PP^{N}_\kappa} = L^*\omega_{\CC\HH_{\kappa'}^{N'}(b)}$.
Fixing $x \in D$, without loss of generality, we assume $F(x)=0, L(x)=0$ by composing the automorphisms on $\CC\PP_{\kappa'}^{N'}(b)$ and $\CC\HH_{\kappa'}^{N'}(b)$. By the standard argument to get rid of $\partial\overline\partial$ (cf. \cite{[CU]}), since there is no non-constant pluriharmonic terms on both sides of the identity, we have
$$ 1+ \sum_{i=1}^{N'-\kappa'} |L_i(z)|^2 - \sum_{j=0}^{\kappa'-1} |L_{N'-j}(z)|^2=  \left( \sum_{i=1}^{N-\kappa} |F_i(z)|^2 - \sum_{j=1}^\kappa |F_{N-\kappa+j}(z)|^2 \right)^{b}.$$

\medskip

The following question was raised in \cite{[CDY], [Y1]}:

\begin{question}
Let $f_1, \cdots, f_m$ be non-constant germs of linearly independent holomorphic functions at $0 \in \mathbb{C}$ such that $f_1(0)=\cdots=f_m(0)=0$ and $a_1, \cdots, a_m$ be non-zero real numbers, do there exist germs of holomorphic functions $g_1, \cdots, g_n$ and non-zero real numbers $b_1, \cdots, b_n$ and $\alpha<0$ such that 
$$1+\sum_{i=1}^m a_i |f_i|^2 = \left( 1 - \sum_{j=1}^n  b_j |g_j|^2 \right)^{\alpha} ?$$
\end{question}

This question was motivated by a classical theorem of Umehara, who showed that when all $a_i, b_j$ are positive, the answer is "no" \cite{[U2]}. In other words, the complex projective space and the complex hyperbolic space does not share a common complex submanifold. As a consequence of Theorem \ref{UA}, we show that under an additional assumption, the question for  the indefinite complex projective space and the indefinite complex hyperbolic space also has a negative
 answer. However, the general question is still open.

\begin{corollary}\label{UA2}
 Let $f_1, \cdots, f_m \in \mathcal{O}_p$ be non-constant germs of holomorphic polynomials at $p$ such that $f_1(p)= \cdots = f_m(p)=0$. Assume $\alpha<0$. For non-zero real numbers $a_1, \cdots, a_m$, then there do not exist $g_1, \cdots, g_n$ and nonzero real numbers $b_1, \cdots, b_n$, such that  $$\left(1+ \sum_{i=1}^m a_i |f_i|^2\right)^{\alpha} = \sum_{j=1}^n b_j |g_j|^2$$ near $p$.
 \end{corollary}


\subsection{K\"ahler manifolds with distinguished K\"ahler potentials}

Consider two real analytic complex manifolds $(M_1, \omega_{M_1})$ and $(M_2, \omega_{M_2})$, with $(M_2, \omega_{M_2}) = (\CC^N_\kappa, \omega_{\CC^N_\kappa})$, $(\CC\PP^N_\kappa(b), \omega_{\CC\PP^N_\kappa(b)})$ or $(\CC\HH^N_\kappa(b), \omega_{\CC\HH^N_\kappa(b)})$. Here $\omega_{M_1}$ is possibly indefinite or degenerate and even can be an arbitrary closed real analytic $(1, 1)$-form. 
 Suppose that $D$ is a complex manifold such that  there exist  holomorphic maps $F=(F_1, \cdots, F_N): D \rightarrow M_1$ and $L=(L_1, \cdots, L_{N'}): D \rightarrow M_2$ with $F^*\omega_{M_1} = L^*\omega_{M_2}$. 
Fixing $p \in D$ and choosing holomorphic coordinate $\{z\}$ near $p$ with $z(p)=0$, without loss of generality, we assume $L(p)=0$ by composing the automorphisms on $M_2$. Let $\{\xi\}$ be the holomorphic coordinate with $\xi(F(p))=0$ and by $\partial\bar\partial$-lemma, $(M_1, \omega_{M_1})$ locally admits a potential function $\varphi$ with $\omega_{M_1} = 4\sqrt{-1} \partial \overline\partial \varphi(\xi, \overline \xi)$. Calabi defined the diastasis function ${\mathcal D}(\xi, \overline \eta) = \varphi(\xi, \overline\xi) + \varphi(\eta, \overline\eta) - \varphi(\xi, \overline\eta) -\varphi(\eta, \overline\xi)$, which is a Hermitian symmetric function without non-constant pluriharmonic terms.
 Moreover, $\omega_{M_1}(\xi) = 4\sqrt{-1} \partial_{\xi} \overline\partial_{\xi} {\mathcal D}(\xi, \overline \eta)$ (cf. \cite{[C]}).

 By the standard argument to get rid of $\partial\overline\partial$ (cf. \cite{[CU]}), since there is no non-constant pluriharmonic terms on both sides of the identity, we have:
$$ \log \left( 1+ \sum_{i=1}^{N-\kappa} |L_i(z)|^2 - \sum_{j=0}^{\kappa-1} |L_{N'-j}(z)|^2 \right)^\mu = {\mathcal D}(F(z), \overline{F(z)})$$ 
when $M_2 = \CC\PP^N_\kappa(b)$ or $\CC\HH^N_\kappa(b)$,
and $$  \sum_{i=\kappa+1}^{N} |L_i(z)|^2 - \sum_{j=1}^{\kappa} |L_{j}(z)|^2  = {\mathcal D}(F(z), \overline{F(z)})   $$
when $M_2 = \CC^N_\kappa$. Note that since $L_j(0)=0$ for all $j$, $$ \left( 1+ \sum_{i=1}^{N-\kappa} |L_i(z)|^2 - \sum_{j=0}^{\kappa-1} |L_{N-j}(z)|^2\right)^\mu \in \tilde\Lambda_p^\mu,~~~~ \sum_{i=\kappa+1}^{N} |L_i(z)|^2 - \sum_{j=1}^{\kappa} |L_{j}(z)|^2  \in \tilde\Lambda_p . $$ 

In general, if the polarization $\varphi(\xi, \overline\eta)$ of the potential function $\varphi$ is Nash algebraic, the diastasis function ${\mathcal D}(\xi, \overline \eta)$  is Nash algebraic as well. Similarly, if $\exp\{\varphi(\xi, \overline \eta)\}$ is Nash algebraic, then
 $\exp\{{\mathcal D}(\xi, \overline \eta)\} = \frac{\exp\{\varphi(\xi, \overline \xi)\} \exp\{\varphi(\eta, \overline \eta)\}}{\exp\{\varphi(\xi, \overline \eta)\}\exp\{\varphi(\eta, \overline \xi)\}}$ is also Nash algebraic.
This fits into the framework of Theorem \ref{UA}.(iii). In fact, this even applies to the product space of $\CC\PP^N_\kappa(b)$ and $\CC\HH^N_\kappa(b)$.
 As a consequence, we have:

\begin{corollary}\label{suf}
Assume $(M_1, \omega_{M_1})$ locally admits a potential function $\varphi(\xi, \overline \eta)$ with $\omega_{M_1} = 4\sqrt{-1} \partial \overline\partial \varphi(\xi, \overline \xi)$.
\begin{enumerate}
\item If  $\varphi(\xi, \overline \eta)$ is  Nash algebraic, then $(M_1, \omega_{M_1})$ and $(M_2, \omega_{M_2})$ do not share common complex submanifolds for  $(M_2, \omega_{M_2}) = \left(
\prod_{j=1}^m N_j , \oplus_{j=1}^m c_j \omega_{N_j} \right)$ with $N_j = \CC\PP^{N_j}_{\kappa_j}(b_j)$ or $\CC\HH^{N_j}_{\kappa_j}(b_j)$.

\item If  $\exp\{\varphi(\xi, \overline \eta)\}$ is Nash algebraic, then $(M_1, \omega_{M_1})$ and $(M_2, \omega_{M_2})$ do not share common complex submanifolds for $(M_2, \omega_{M_2}) =  \left(
\prod_{j=1}^m N_j , \oplus_{j=1}^m c_j \omega_{N_j} \right)$ with $N_j = \CC^{N_j}_{\kappa_j}(b_j)$.
\end{enumerate}
\end{corollary}

\begin{corollary}\label{ex}
The indefinite complex Euclidean spaces and the bounded domain with Nash algebraic Bergman kernel function equipped with the Bergman metric do not have common complex submanifolds.
\end{corollary}

Note that the following bounded domains have Nash algebraic Bergman kernel function: (a) bounded 
 homogeneous domains (cf. Proposition 1 in \cite{[CH]}), (b) the Hartogs triangle and  its generalizations \cite{[Ed]}, (c) the minimal ball \cite{[OPY]}, (d) the symmetrized polydisc \cite{[EW], [CKY]} and (e) certain Hartogs domains over bounded homogeneous domains (for instance, Cartan-Hartogs domains) \cite{[IPY]}.
Therefore, they do not share common complex submanifolds with the indefinite complex Euclidean spaces.

\begin{example}
Let $\tilde \varphi(\xi, \eta)$ be any Nash algebraic, irrational function in $\mathbb{C}^2$ such that $\tilde \varphi(\xi, \bar \eta)$ is Hermitian symmetric. Let $D$ be a planar domain such that 
$\tilde \varphi(\xi, \bar \xi)$ is $C^2$-smooth up to the boundary of $D$. Then there exists $C>0$, such that $\varphi(\xi, \bar \xi)= C |\xi|^2 + \tilde \varphi(\xi, \bar \xi)$ is a strictly plurisubharmonic function in $D$. $\varphi(\xi, \bar \xi)$ is thus an example of Nash algebraic and irrational potential function and $\sqrt{-1} \partial \overline\partial \varphi(\xi, \overline \xi)$ provides a K\"ahler metric on $D$. This example satisfies the assumption in Part 1 of Corollary \ref{suf}. One can construct many other examples of K\"ahler metrics on domains in $\mathbb{C}^n$ with Nash algebraic and irrational potential functions in a similar manner. 
\end{example}

\begin{remark}
It follows from Theorem 3.1 in \cite{[LM2]} that, there exist rational functions $h_1, \cdots,  h_m$ and positive numbers $\mu_1, \cdots, \mu_m$ such that $\log \left(\prod_{j=1}^m h_j^{\mu_j}(z) \right)$ is the potential function of the homogeneous K\"ahler metric on the bounded homogeneous domain. By polarization, Calabi's diastasis function fits into the form of $\log H(z, w)$ in Theorem  \ref{UA}.(ii). As a consequence of Theorem  \ref{UA}.(ii)(2), a bounded homogeneous domain equipped with a homogeneous K\"ahler metric and an indefinite complex Euclidean space do not share common complex submanifolds. This recovers Theorem 1.1.(ii) in \cite{[LM2]}.
\end{remark}

{\bf Statement}:
Data sharing not applicable to this article as no datasets were generated or analysed during the current study.


\end{document}